\documentclass[11pt]{article}
\usepackage[utf8]{inputenc}
\usepackage[margin=1.1in]{geometry}
\usepackage{amsmath,amssymb,amsthm,mathtools,mathrsfs}
\usepackage[pdfpagelabels,colorlinks,linkcolor=blue,citecolor=blue]{hyperref}
\usepackage[capitalize,nameinlink]{cleveref}
\usepackage[shortlabels]{enumitem}
\usepackage{booktabs}
\usepackage{graphicx}
\usepackage{subcaption}

\newtheorem{theorem}{Theorem}[section]
\newtheorem{lemma}[theorem]{Lemma}
\newtheorem{corollary}[theorem]{Corollary}
\theoremstyle{definition}
\newtheorem{definition}[theorem]{Definition}
\newtheorem{example}[theorem]{Example}
\theoremstyle{remark}
\newtheorem{remark}[theorem]{Remark}

\crefformat{equation}{\textup{#2(#1)#3}}
\Crefformat{equation}{#2Equation~\textup{(#1)}#3}
\crefdefaultlabelformat{#2\textup{#1}#3}

\newcommand{\norm}[1]{\left\lVert#1\right\rVert}
\newcommand{\scal}[2]{\left\langle{#1},{#2}\right\rangle}
\DeclareMathOperator{\Id}{Id}
\DeclareMathOperator{\aff}{aff}
\DeclareMathOperator{\dist}{dist}
\DeclareMathOperator{\inte}{int}
\DeclareMathOperator{\circum}{circ}
\usepackage{bbm}
\newcommand{\re}{\mathbbm{R}}
\newcommand{\na}{\mathbbm{N}}
\newcommand{\cCRM}{\ensuremath{\mathrm{cCRM}}}
\newcommand{\pCRMOp}{\ensuremath{\mathscr{C}}}
\allowdisplaybreaks

\begin{document}
\title{\textbf{Q-quadratic convergence of the centralized circumcentered-reflection method under a relative interior condition}\thanks{The author was supported by the National Science
        Foundation, Grant DMS-2307328, and by an internal
        grant from NIU.}}
\author{Yunier Bello-Cruz\\[2pt]
{\small Department of Mathematical Sciences,
        Northern Illinois University}\\
{\small DeKalb, IL 60115, USA\quad
        \texttt{yunierbello@niu.edu}}}
\date{\small arXiv:2604.11450}
\maketitle

\begin{abstract}
\noindent The centralized circumcentered-reflection method
(\cCRM) of Behling, Bello-Cruz, Iusem, and
Santos~\cite{Behling:2024} is known to converge
superlinearly for the feasibility problem
$\operatorname{find}\;z\in X\cap Y$ under a $\mathcal{C}^1$
smoothness assumption on the boundaries of $X$ and $Y$.
We sharpen this to a quantitative rate: when the boundaries
are $\mathcal{C}^2$ near the limit point $\bar z$, \cCRM\
converges Q-quadratically, with an asymptotic constant
\(
  2\max(\kappa_X,\kappa_Y)/\omega
\)
governed by the boundary curvatures $\kappa_X,\kappa_Y$ at
$\bar z$ and the local error-bound modulus $\omega$. The
estimate matches Newton-type second-order behavior even
though \cCRM\ uses only projections and circumcenters, and
numerical experiments on equality-constrained and spectral
feasibility problems exhibit the predicted quadratic rate,
with \cCRM\ reaching machine precision in a handful of steps
where alternating projections and Douglas--Rachford take many.
The argument is local and does not require $X\cap Y$ to have
nonempty interior in $\re^n$: it suffices that the sets share
an affine hull $L=\aff(X)=\aff(Y)$ and meet with nonempty
relative interior, which is the natural setting for
equality-constrained and spectral feasibility problems, where
the classical full-dimensional hypothesis necessarily fails.
A $\mathcal{C}^1$ version of the argument recovers and extends
the superlinear rate of~\cite{Behling:2024} to this
lower-dimensional regime. The case $\aff(X)\neq\aff(Y)$ is
identified as open.

\medskip\noindent
\textbf{Keywords:} Convex feasibility $\cdot$
circumcentered-reflection method $\cdot$
Q-quadratic convergence $\cdot$ boundary curvature $\cdot$
relative interior.

\medskip\noindent\textbf{MSC:} 49M27 $\cdot$ 65K05 $\cdot$
90C25.
\end{abstract}

\section{Introduction}\label{sec:intro}

The circumcentered-reflection method was introduced
in~\cite{Behling:2018b} as an acceleration of the
Douglas--Rachford method for convex feasibility: rather than
taking the next iterate from the usual reflection step, one
computes the circumcenter of three related reflected points.
The approach extends to a centralized variant (\cCRM) that is
globally convergent and, under suitable conditions, converges
faster than alternating projections~\cite{Arefidamghani:2021}.
Circumcenter schemes have since been extended to nonconvex
feasibility problems~\cite{Dizon:2022} and modified to achieve
finite convergence under a Slater condition~\cite{Behling:2024b}.

The fastest rate established for \cCRM\ to date is
superlinear, proved by Behling, Bello-Cruz, Iusem, and
Santos~\cite{Behling:2024} under a $\mathcal{C}^1$ assumption
on the boundaries. Superlinear convergence is a qualitative
statement: it says the error ratios tend to zero but says
nothing about how fast, and it carries no constant one can
point to. Our main result is quantitative. We show that one
more degree of boundary smoothness pins the rate down
completely: when $\partial X$ and $\partial Y$ are
$\mathcal{C}^2$ near the limit point, \cCRM\ is
Q-quadratic, and the asymptotic constant is an explicit
function of the boundary curvatures and the local error-bound
modulus (Theorem~\ref{thm:main}). The mechanism is
geometric, second-order contact between each boundary and
its tangent hyperplane, and it produces Newton-type behavior
from an iteration that computes nothing beyond projections and
circumcenters.

To reach the structured problems where this matters most we
also remove a dimensional restriction in~\cite{Behling:2024}.
That result assumes $X\cap Y$ has nonempty interior in
$\re^n$, which rules out every equality-constrained problem:
once linear constraints $Az=b$ are present, both sets lie in a
proper affine subspace and the interior is empty. We carry the
analysis out within the common affine hull instead, so the
results apply verbatim to semidefinite and spectral
feasibility problems with equality constraints. The
$\mathcal{C}^1$ version of the argument
(Corollary~\ref{cor:super}) recovers superlinear convergence
in this setting; the $\mathcal{C}^2$ version
(Theorem~\ref{thm:main}) gives the quadratic rate.

The \cCRM\ finds a point in $X\cap Y$ for closed convex sets
$X,Y\subset\re^n$ with $X\cap Y\neq\emptyset$.
Starting from $z^0\in\re^n$, it produces $z^{k+1}=T(z^k)$,
where
\begin{equation}\label{eq:cCRM}
  T(z) \coloneqq \pCRMOp(z_C)
       = \circum\!\bigl\{z_C,\,R_X(z_C),\,R_Y(z_C)\bigr\},
\end{equation}
$R_X\coloneqq 2P_X-\Id$, $R_Y\coloneqq 2P_Y-\Id$,
$z_C\coloneqq\tfrac{1}{2}(P_YP_X(z)+P_X(P_YP_X(z)))$,
and $\circum\{a,b,c\}$ denotes the circumcenter of
$(a,b,c)$~\cite[Eq.~(2)]{Behling:2018a}.
The point $z_C$ is the centralized point: one step of
alternating projections followed by an averaging.
The sequence is F\'{e}jer monotone and converges globally;
under a local error bound it converges
linearly~\cite{Bauschke:1996,Behling:2024}.
The best known rate is the following.

\begin{theorem}[{\cite[Theorem~5.3]{Behling:2024}}]%
\label{thm:orig}
Let $(z^k)_{k\in\na}$ be generated by~\cref{eq:cCRM} and
converge to $\bar z\in X\cap Y$. If
\begin{enumerate}[\upshape(H1),leftmargin=3em,align=right,nosep]
  \item\label{H1} $\inte(X\cap Y)\neq\emptyset$, and
  \item\label{H2} $\partial X$ and $\partial Y$ are
    $\mathcal{C}^1$ manifolds of dimension $n-1$ near $\bar z$,
\end{enumerate}
then $(z^k)_{k\in\na}$ converges to $\bar z$ superlinearly.
\end{theorem}

The proof relies on the identity
$T(z^k)=P_{H_X^k\cap H_Y^k}(z_C^k)$~%
\cite[Lemma~3.3]{Behling:2024}, where $H_X^k$ and $H_Y^k$
are the hyperplanes tangent to $\partial X$ and $\partial Y$
at $P_X(z_C^k)$ and $P_Y(z_C^k)$.
The $\mathcal{C}^1$ contact between a manifold and its tangent
hyperplane forces
$\dist(T(z^k),X)=o(\dist(z^k,X\cap Y))$; superlinearity
follows from the Bauschke--Borwein error
bound~\cite{Bauschke:1993}.

Both conditions fail for a large class of problems.
Condition~\ref{H1} requires $X\cap Y$ to have nonempty
interior in $\re^n$, forcing $X$ and $Y$ to be
full-dimensional.
This fails whenever both sets are confined to a proper affine
subspace $L\subsetneq\re^n$: in equality-constrained
feasibility, the constraints $Az=b$ force both sets into
$L=\{Az=b\}$; in semidefinite feasibility, the linear map
$\mathcal{A}(\Sigma)=b$ defines a proper affine subspace
of $\mathbb{S}^n$.
In all such cases $\inte_{\re^n}(X\cap Y)=\emptyset$, so
Theorem~\ref{thm:orig} is inapplicable regardless of how
smooth or well-separated $X$ and $Y$ are within $L$.
Condition~\ref{H2} is equally mismatched: it asks for
boundaries that are hypersurfaces in $\re^n$, which is the
wrong notion when the sets live in a lower-dimensional
subspace.

We replace interior by relative interior throughout and
measure boundary smoothness within $\aff(X)$ and $\aff(Y)$. Write $d_X=\dim\aff(X)$ and $d_Y=\dim\aff(Y)$,
and set:
\begin{enumerate}[\upshape(H1$'$),leftmargin=3em,align=right]
  \item\label{H1prime}
    $\operatorname{ri}(X)\cap\operatorname{ri}(Y)\neq\emptyset$;
\end{enumerate}
\begin{enumerate}[\upshape(H2$'$),leftmargin=3em,align=right]
  \item\label{H2prime}
    $\partial_{\aff(X)}X$ and $\partial_{\aff(Y)}Y$ are
    $\mathcal{C}^1$ manifolds of dimension $d_X-1$ and
    $d_Y-1$ near $\bar z$ within $\aff(X)$ and $\aff(Y)$;
\end{enumerate}
\begin{enumerate}[\upshape(H2$''$),leftmargin=3em,align=right]
  \item\label{H2pp}
    $\partial_{\aff(X)}X$ and $\partial_{\aff(Y)}Y$ are
    $\mathcal{C}^2$ manifolds of dimension $d_X-1$ and
    $d_Y-1$ near $\bar z$ within $\aff(X)$ and $\aff(Y)$.
\end{enumerate}
When $d_X=d_Y=n$ these reduce to \ref{H1}--\ref{H2}.
Condition~\ref{H1prime} makes $\{X,Y\}$ boundedly linearly
regular: for two closed convex sets in $\re^n$ with
$\operatorname{ri}(X)\cap\operatorname{ri}(Y)\neq\emptyset$,
the error bound~\cref{eq:EB} holds with some $\omega>0$ on a
neighborhood of any point of
$X\cap Y$~\cite[Theorem~5.7]{Bauschke:1993,Bauschke:1996},
so the projection iterates converge linearly.

Both results require $\aff(X)=\aff(Y)=:L$ and impose their
smoothness and regularity conditions at the limit point
$\bar z$. Theorem~\ref{thm:main} is the main contribution:
under \ref{H1prime} and the $\mathcal{C}^2$
condition~\ref{H2pp} it gives Q-quadratic convergence with
the asymptotic constant
$2\max(\kappa_X,\kappa_Y)/\omega$, where $\kappa_X,\kappa_Y$
are the curvatures of $\partial_L X$ and $\partial_L Y$ at
$\bar z$ and $\omega$ is the local error-bound modulus
(precise definitions in Section~\ref{sec:prelim}).
Corollary~\ref{cor:super} is the $\mathcal{C}^1$ companion,
giving superlinear convergence under \ref{H1prime}
and~\ref{H2prime}. Even in the full-dimensional case
$d_X=d_Y=n$, where~\ref{H1prime} reduces to the interior
condition of~\cite{Behling:2024}, the result is new:
$\mathcal{C}^2$ boundaries improve the superlinear rate
of~\cite[Theorem~5.3]{Behling:2024} to Q-quadratic with an
explicit constant. The case $\aff(X)\neq\aff(Y)$ is open; the
obstruction is identified in Section~\ref{sec:disc}.
Section~\ref{sec:applications} collects examples and
structured applications, and Section~\ref{sec:numerics}
reports numerical experiments that exhibit the quadratic rate
and compare \cCRM\ with alternating projections and
Douglas--Rachford.

\section{Preliminaries}\label{sec:prelim}

For a nonempty convex set $C\subset\re^n$, let $\aff(C)$
denote its affine hull with direction subspace
$V_C\coloneqq\aff(C)-\aff(C)$, so $\aff(C)=\bar z+V_C$ for
any $\bar z\in C$.
The relative interior $\operatorname{ri}(C)$ is the interior
of $C$ within $\aff(C)$; it is nonempty for every nonempty
convex $C$, even when $\inte_{\re^n}(C)=\emptyset$.
The relative boundary is
$\partial_{\aff(C)}C\coloneqq\overline{C}\setminus
\operatorname{ri}(C)$.
We write $\inte_L(C)$ for the interior of $C$ within an
affine subspace $L$.
For real-valued functions $f$ and $g>0$, we write
$f=O(g)$ as $t\to 0$ if $\limsup_{t\to 0}|f(t)|/g(t)<\infty$,
and $f=o(g)$ as $t\to 0$ if $\lim_{t\to 0}f(t)/g(t)=0$.

When $\aff(X)=\aff(Y)=L$ and both sets are full-dimensional
in $L$, one has $\operatorname{ri}(X)=\inte_L(X)$ and
$\operatorname{ri}(Y)=\inte_L(Y)$, so \ref{H1prime} implies
$\inte_L(X\cap Y)\neq\emptyset$.

\begin{definition}[$\mathcal{C}^k$ relative boundary]%
\label{def:ck_boundary}
Let $C\subset\re^n$ be closed convex with
$d\coloneqq\dim\aff(C)\geq 1$.
We say $\partial_{\aff(C)}C$ is a $\mathcal{C}^k$ manifold
of dimension $d-1$ near $\bar z$ if there exist a
neighborhood $U$ of $\bar z$ in $\aff(C)$ and
$g:U\to\re$ of class $\mathcal{C}^k$ such that
\[
  \partial_{\aff(C)}C\cap U = \{z\in U : g(z)=0\},\qquad
  C\cap U = \{z\in U : g(z)\leq 0\},
\]
and $\nabla_{V_C}g(\bar z)\neq 0$, where
$\nabla_{V_C}g(\bar z)$ is the gradient of $g$ at $\bar z$
within $\aff(C)$, i.e., the unique vector in $V_C$ such that
$\scal{\nabla_{V_C}g(\bar z)}{v}=Dg(\bar z)[v]$ for all
$v\in V_C$.
\end{definition}

At a regular boundary point $\bar z$, the tangent space
within $\aff(C)$ is
\[
  T_C(\bar z)\coloneqq\bigl\{v\in V_C:
    \scal{\nabla_{V_C}g(\bar z)}{v}=0\bigr\}.
\]

\begin{example}[Boundary regularity]\label{ex:smooth_bdry}
The following instances illustrate when conditions~\ref{H2prime}
and~\ref{H2pp} hold or fail, and in the smooth cases give the
curvature $\kappa_C$ explicitly.
\begin{enumerate}[(a),nosep]
  \item \textit{Ellipsoids.}
    For $A\succ 0$, $C=\{z:z^\top\! Az\leq 1\}$ has
    $g=z^\top\! Az-1$ with $\nabla g(\bar z)=2A\bar z\neq 0$
    on $\partial C$, so both~\ref{H2prime} and~\ref{H2pp}
    hold globally. The principal curvatures at $\bar z$ are
    the eigenvalues of $A|_{T_C(\bar z)}/\norm{A\bar z}$;
    for a ball of radius $r$ this gives $\kappa_C=1/r$
    everywhere.
  \item \textit{Second-order cone.}
    $\mathcal{K}=\{(t,u):\norm{u}\leq t\}$ has
    $\mathcal{C}^\infty$ boundary on $\{t>0\}$, where both
    conditions hold. At the apex $(0,0)$ the boundary is not
    a manifold and~\ref{H2prime} fails.
  \item \textit{Spectral sets.}
    The set
    $\{\Sigma\in\mathbb{S}^n:\lambda_{\max}(\Sigma)\leq a\}
    \cap\{\operatorname{tr}(\Sigma)=1\}$ has a $\mathcal{C}^2$
    relative boundary at any $\bar\Sigma$ where
    $\lambda_{\max}(\bar\Sigma)=a$ is a simple
    eigenvalue~\cite{Lewis:1996}.
  \item \textit{Polytopes.}
    Condition~\ref{H2prime} holds at
    $\bar z\in\partial_{\aff(C)}C$ if and only if $\bar z$
    lies in the relative interior of a facet; at edges and
    vertices the boundary is not a manifold
    and~\ref{H2prime} fails.
\end{enumerate}
\end{example}

When $\partial_{\aff(C)}C$ is $\mathcal{C}^2$ near $\bar z$,
the tangent space $T_C(\bar z)$ is well-defined and the
\emph{curvature} $\kappa_C$ at $\bar z$ is the spectral
radius of the shape operator of $\partial_{\aff(C)}C$,
which in terms of the local representation $g$ reads
\begin{equation}\label{eq:kappa_def}
  \kappa_C \coloneqq
  \max_{\substack{v\in T_C(\bar z)\\\norm{v}=1}}
  \frac{\bigl|D^2_{V_C}g(\bar z)[v,v]\bigr|}
       {\norm{\nabla_{V_C}g(\bar z)}},
\end{equation}
where $D^2_{V_C}g(\bar z)$ is the Hessian of $g$ at $\bar z$
restricted to $V_C\times V_C$, so that
$D^2_{V_C}g(\bar z)[v,v]$ is the second directional
derivative of $g$ along $v\in V_C$.
The value $\kappa_C$ is independent of the choice of $g$,
since it equals the spectral radius of the second
fundamental form of $\partial_{\aff(C)}C$ at $\bar z$, an
intrinsic invariant of the hypersurface within $\aff(C)$.
For a ball of radius $r$, $\kappa_C=1/r$ everywhere on
$\partial C$.

The key estimate connecting boundary curvature to the
distance from the tangent hyperplane is the following.

\begin{lemma}[Tangent hyperplane distance estimate]%
\label{lem:curvature}
Let $C$ satisfy~\ref{H2pp} at $\bar z$. There exist a
neighborhood $U$ of $\bar z$ in $\aff(C)$ and a constant
$c>0$ such that for every regular boundary point
$p\in\partial_{\aff(C)}C\cap U$ with tangent hyperplane
$H\subset\aff(C)$ at $p$,
\begin{equation}\label{eq:curv_est}
  \dist(w,C) \leq \tfrac{1}{2}\kappa_C\,\norm{w-p}^2
              + o\!\bigl(\norm{w-p}^2\bigr)
  \qquad \forall\,w\in H,\ \norm{w-p}\to 0,
\end{equation}
where the $o(\cdot)$ term is uniform for $p\in U$.
In particular, for every $\eta>0$ there is a neighborhood on
which
$\dist(w,C)\leq\bigl(\tfrac{1}{2}\kappa_C+\eta\bigr)\norm{w-p}^2$.
\end{lemma}

\begin{proof}
Work in $V_C$ via the chart of
Definition~\ref{def:ck_boundary}, so $g\in\mathcal{C}^2$ on
$U$ with $\nabla_{V_C}g(p)\neq 0$ and, by continuity and
compactness, $\norm{\nabla_{V_C}g(p)}\geq c>0$ for $p$ in a
possibly smaller neighborhood. Fix such a $p$ and let
$w\in H$, so $\scal{\nabla_{V_C}g(p)}{w-p}=0$. Since
$g(p)=0$, Taylor's theorem with the integral remainder gives
\begin{equation}\label{eq:taylor}
  g(w)=\tfrac{1}{2}D^2_{V_C}g(p)[w-p,w-p]
       +\rho(p,w),\qquad
  \rho(p,w)=o\!\bigl(\norm{w-p}^2\bigr)
\end{equation}
uniformly in $p\in U$, because $D^2_{V_C}g$ is uniformly
continuous on the compact closure of $U$. As $w-p\in T_C(p)$
(it is orthogonal to $\nabla_{V_C}g(p)$), the definition
of $\kappa_C$ in~\cref{eq:kappa_def} yields
\[
  \bigl|D^2_{V_C}g(p)[w-p,w-p]\bigr|
  \leq \kappa_C\,\norm{\nabla_{V_C}g(p)}\,\norm{w-p}^2 .
\]
Next we bound $\dist(w,C)$ above by the scaled function
value. Assume $w\notin C$, so $g(w)>0$. Let
$\nu\coloneqq\nabla_{V_C}g(p)/\norm{\nabla_{V_C}g(p)}$ and
descend from $w$ along $-\nu$: set
$\psi(t)\coloneqq g(w-t\nu)$ for $t\geq 0$. Then
$\psi(0)=g(w)$ and, since $g\in\mathcal{C}^1$ with
$\nabla_{V_C}g$ continuous,
$\psi'(t)=-\scal{\nabla_{V_C}g(w-t\nu)}{\nu}
=-\norm{\nabla_{V_C}g(p)}\,(1+o(1))$ uniformly for $w,t$
small, because $\nabla_{V_C}g(w-t\nu)\to\nabla_{V_C}g(p)$.
Hence $\psi$ vanishes at some
$t^*=g(w)/\norm{\nabla_{V_C}g(p)}\,(1+o(1))$, and the point
$w-t^*\nu$ lies on $\partial_{\aff(C)}C$, so
\begin{equation}\label{eq:dist_bound}
  \dist(w,C)\leq t^*
  =\frac{g(w)_+}{\norm{\nabla_{V_C}g(p)}}\,(1+o(1))
  \qquad(w\to\bar z).
\end{equation}
Combining~\cref{eq:dist_bound} with~\cref{eq:taylor} and the
curvature bound, and using $g(w)_+=O(\norm{w-p}^2)$,
\[
  \dist(w,C)
  \leq\frac{\tfrac12\bigl|D^2_{V_C}g(p)[w-p,w-p]\bigr|
       +o(\norm{w-p}^2)}{\norm{\nabla_{V_C}g(p)}}
  \leq\tfrac12\kappa_C\norm{w-p}^2+o(\norm{w-p}^2),
\]
which is~\cref{eq:curv_est}; the $o(\cdot)$ term is uniform
in $p\in U$ because every estimate above is. For the final
statement, fix $\eta>0$ and choose the neighborhood small
enough that the uniform $o(\norm{w-p}^2)$ term is at most
$\eta\norm{w-p}^2$ there; then
$\dist(w,C)\leq(\tfrac12\kappa_C+\eta)\norm{w-p}^2$ for all
such $w\in H$ and all regular $p\in U$.
\end{proof}

\begin{remark}\label{rem:halffactor}
The factor $\tfrac{1}{2}$ in~\cref{eq:curv_est} is sharp:
for a Euclidean ball of radius $r$ ($\kappa_C=1/r$) and
$w\in H$, $\dist(w,C)=\sqrt{r^2+\norm{w-p}^2}-r
=\tfrac{1}{2r}\norm{w-p}^2+O(\norm{w-p}^4)$, matching
$\tfrac{1}{2}\kappa_C\norm{w-p}^2$ exactly to leading order.
Lemma~\ref{lem:curvature} is applied below with $w=T(z^k)$
and $p=P_C(z_C^k)$, where $\norm{w-p}$ carries an additional
factor absorbed into the asymptotic constant of
Theorem~\ref{thm:main}.
\end{remark}

The next two definitions fix the convergence terminology
used throughout the paper.

\begin{definition}[F\'{e}jer monotonicity]\label{def:fejer}
A sequence $(w^k)_{k\in\na}\subset\re^n$ is
\emph{F\'{e}jer monotone} with respect to $M\subset\re^n$
if $$\norm{w^{k+1}-s}\leq\norm{w^k-s},$$ for all $s\in M$
and all $k\in\na$.
\end{definition}

F\'{e}jer monotone sequences are bounded and the sequence
$(\dist(w^k,M))_{k\in\na}$ is monotonically decreasing,
hence convergent~\cite[Proposition~5.4]{Bauschke:2017a}.
If $(w^k)_{k\in\na}$ has a cluster point in $M$, then the
whole sequence converges to that
point~\cite[Proposition~5.9]{Bauschke:2017a}.
Moreover, if $(w^k)_{k\in\na}$ converges to $\bar w\in M$,
then
\begin{equation}\label{eq:fejer_bound}
  \norm{w^k-\bar w}\leq 2\,\dist(w^k,M)
  \quad\forall\,k\in\na.
\end{equation}
Indeed, let $\bar w^k\coloneqq P_M(w^k)$; F\'{e}jer monotonicity
with $s=\bar w^k$ gives $\norm{w^j-\bar w^k}\leq\norm{w^k-\bar w^k}$ for
all $j\geq k$, and taking $j\to\infty$ yields
$\norm{\bar w-\bar w^k}\leq\dist(w^k,M)$.
The triangle inequality then gives~\cref{eq:fejer_bound}.

\begin{definition}[Convergence rates]\label{def:rates}
Let $(w^k)_{k\in\na}\subset\re^n$ converge to $\bar w$.
\begin{enumerate}[(i),nosep]
  \item $(w^k)_{k\in\na}$ converges \emph{Q-linearly} with
    constant $c\in(0,1)$ if
    \[
      \limsup_{k\to\infty}
      \frac{\norm{w^{k+1}-\bar w}}{\norm{w^k-\bar w}}
      \leq c.
    \]
  \item $(w^k)_{k\in\na}$ converges \emph{superlinearly} if
    \[
      \lim_{k\to\infty}
      \frac{\norm{w^{k+1}-\bar w}}{\norm{w^k-\bar w}}=0.
    \]
  \item $(w^k)_{k\in\na}$ converges \emph{Q-quadratically}
    with \emph{asymptotic constant} $C\geq 0$ if
    \[
      \limsup_{k\to\infty}
      \frac{\norm{w^{k+1}-\bar w}}{\norm{w^k-\bar w}^2}
      \leq C<\infty.
    \]
\end{enumerate}
Superlinear convergence is strictly faster than Q-linear;
Q-quadratic convergence implies superlinear.
\end{definition}

\begin{lemma}[{\cite[Lemmas~3.4--3.6,
Theorems~4.1 and~4.3]{Behling:2024}}]%
\label{lem:background}
Let $X,Y\subset\re^n$ be closed convex with
$X\cap Y\neq\emptyset$.
For all $s\in X\cap Y$:
\begin{enumerate}[\upshape(i),nosep]
  \item $\norm{T(z)-s}^2\leq\norm{z-s}^2
    -\tfrac{1}{8}\norm{z-T(z)}^2$;
  \item $\norm{T(z)-s}
    \leq\norm{z_C-s}
    \leq\norm{P_YP_X(z)-s}
    \leq\norm{z-s}$;
  \item $T(z)=P_{H_X\cap H_Y}(z_C)$, where $H_X$ and $H_Y$
    are the hyperplanes through $P_X(z_C)$ and $P_Y(z_C)$
    orthogonal to $z_C-P_X(z_C)$ and $z_C-P_Y(z_C)$,
    respectively.
\end{enumerate}
For any $z^0\in\re^n$, $(z^k)_{k\in\na}$ is F\'{e}jer
monotone with respect to $X\cap Y$ and converges to some
$\bar z\in X\cap Y$.
Under~\ref{H1prime} the pair $\{X,Y\}$ is boundedly
linearly regular near $\bar z$~\cite[Theorem~5.7]{Bauschke:1993},
which is to say it satisfies the local error bound
\begin{equation}\label{eq:EB}\tag{EB}
  \omega\,\dist(z,X\cap Y)
  \leq\max\!\bigl\{\dist(z,X),\,\dist(z,Y)\bigr\}
\end{equation}
for all $z$ in some neighborhood of $\bar z$. We take
$\omega$ to be the largest such constant, namely the local
\emph{Bauschke--Borwein regularity modulus}
\begin{equation}\label{eq:omega_def}
  \omega \coloneqq
  \liminf_{\substack{z\to\bar z\\ z\notin X\cap Y}}
  \frac{\max\{\dist(z,X),\dist(z,Y)\}}{\dist(z,X\cap Y)}
  \in(0,1].
\end{equation}
Positivity of $\omega$ is exactly bounded linear
regularity~\cite[Theorem~5.7]{Bauschke:1993,Bauschke:1996};
the upper bound $\omega\leq 1$ is automatic, since
$X\cap Y\subseteq X$ and $X\cap Y\subseteq Y$ force
$\max\{\dist(z,X),\dist(z,Y)\}\leq\dist(z,X\cap Y)$.
When $X$ and $Y$ are subspaces, $\omega$ reduces to a
function of the Friedrichs angle between them, and in
general it is the quantitative regularity constant that
governs the linear rate of projection methods on
$\{X,Y\}$. It is this $\omega$ that appears in the rate
of Theorem~\ref{thm:main}.
\end{lemma}

\section{Superlinear and Q-quadratic convergence}%
\label{sec:main}

Throughout this section $\aff(X)=\aff(Y)=L$ with $d=\dim L$.
Since $X,Y\subset L$, both $P_X$ and $P_Y$ map into $L$;
reflections preserve $L$; and the circumcenter of a triple
in $L$ lies in $L$.
Hence $z^1\in L$, and by induction $z^k\in L$ for all
$k\geq 1$.

The reduction to $\re^d$ rests on the equivariance of all
three operators defining $T$ under affine isometries.

\begin{lemma}[Equivariance of \cCRM\ under affine isometries]%
\label{lem:equivariance}
Let $\phi:L\to\re^d$, $\phi(z)=Q^\top(z-z_0)$, be an affine
isometry, where $z_0\in L$ and the columns of
$Q\in\re^{n\times d}$ form an orthonormal basis of $V_L$.
For every nonempty closed convex $C\subset L$ and every
$z\in L$:
\begin{enumerate}[\upshape(i),nosep]
  \item $P_{\phi(C)}(\phi(z))=\phi(P_C(z))$ and
        $R_{\phi(C)}(\phi(z))=\phi(R_C(z))$;
  \item $\circum\{\phi(a),\phi(b),\phi(c)\}
        =\phi(\circum\{a,b,c\})$ for every triple
        $a,b,c\in L$ whose circumcenter is defined;
  \item consequently $T_{\phi(X),\phi(Y)}(\phi(z))
        =\phi\bigl(T_{X,Y}(z)\bigr)$, and $\phi$ maps the
        \cCRM\ sequence for $(X,Y)$ starting at $z^1\in L$
        to the \cCRM\ sequence for $(\phi(X),\phi(Y))$
        starting at $\phi(z^1)$.
\end{enumerate}
\end{lemma}

\begin{proof}
The map $\phi$ is a bijective isometry from the affine space
$L$ onto $\re^d$: for $u,v\in L$, $u-v\in V_L$ and
$\norm{\phi(u)-\phi(v)}=\norm{Q^\top(u-v)}=\norm{u-v}$ since
$Q$ has orthonormal columns spanning $V_L$. Its inverse is
$\phi^{-1}(y)=z_0+Qy$.

(i) Fix $z\in L$. For every $w\in C\subset L$ we have
$\norm{\phi(z)-\phi(w)}=\norm{z-w}$, so $\phi$ matches the
objective $\norm{z-\,\cdot\,}$ minimized over $C$ with the
objective $\norm{\phi(z)-\,\cdot\,}$ minimized over
$\phi(C)$. Since $\phi$ is a bijection onto $\re^d$ and
$\phi(C)$ is closed and convex, the unique minimizer over
$\phi(C)$ is the image of the unique minimizer over $C$,
i.e.\ $P_{\phi(C)}(\phi(z))=\phi(P_C(z))$. For the
reflection, write $R_C(z)=2P_C(z)-z$; the coefficients
$2$ and $-1$ sum to $1$, so this is an affine combination of
$P_C(z)$ and $z$, and an affine map commutes with affine
combinations. Hence
$R_{\phi(C)}(\phi(z))=2\phi(P_C(z))-\phi(z)
=\phi\bigl(2P_C(z)-z\bigr)=\phi(R_C(z))$.

(ii) By definition~\cite[Eq.~(2)]{Behling:2018a},
$\circum\{a,b,c\}$, when it exists, is the unique point
$p\in\aff\{a,b,c\}$ with
$\norm{p-a}=\norm{p-b}=\norm{p-c}$; this characterization
makes no reference to the dimension of $\aff\{a,b,c\}$ and so
covers collinear and coincident triples alike. Because
$\phi$ is an affine isometry, it maps $\aff\{a,b,c\}$
bijectively onto $\aff\{\phi(a),\phi(b),\phi(c)\}$ and
preserves all three distances, so $\phi(p)$ is the unique
equidistant point of the image triple in its affine hull;
that is, $\circum\{\phi(a),\phi(b),\phi(c)\}=\phi(p)$. In
particular the image circumcenter exists exactly when the
original does.

(iii) The point $z_C=\tfrac12\bigl(P_YP_X(z)
+P_X(P_YP_X(z))\bigr)$ is an affine combination (coefficients
$\tfrac12,\tfrac12$) of points produced by the projections
in~(i), so $\phi(z_C)=(\phi(z))_C$ for the image pair.
Applying~(i) to $R_X,R_Y$ and~(ii) to the resulting triple,
all of which lie in $L$ and have a defined circumcenter by
the well-posedness of \cCRM~\cite{Behling:2018a,Behling:2024},
gives $T_{\phi(X),\phi(Y)}(\phi(z))=\phi(T_{X,Y}(z))$. The
final claim follows by induction on $k$, the base case being
$z^1\in L$.
\end{proof}

Because $X$ and $Y$ are full-dimensional in $L$,
condition~\ref{H1prime} gives $\inte_L(X\cap Y)\neq\emptyset$,
so by Lemma~\ref{lem:equivariance} the image pair
$(\phi(X),\phi(Y))$ in $\re^d$ satisfies
$\inte(\phi(X)\cap\phi(Y))\neq\emptyset$, and its boundaries
have the same smoothness as $\partial_L X$ and $\partial_L Y$
because $\phi$ is a $\mathcal{C}^\infty$ diffeomorphism.

One point on the logical order. By
Lemma~\ref{lem:background}, the sequence $(z^k)_{k\in\na}$ is
F\'{e}jer monotone with respect to $X\cap Y$ and converges to
some $\bar z\in X\cap Y$ from any starting point, with no
smoothness assumed. The conditions \ref{H2prime}
and~\ref{H2pp} below are imposed at this limit, and the
curvatures $\kappa_X,\kappa_Y$ and the modulus $\omega$ are
evaluated there. The local geometry is thus constrained only
at the point the iteration selects, not at every candidate
solution.

\begin{corollary}[Superlinear convergence]\label{cor:super}
Let $X,Y\subset\re^n$ be closed convex with
$X\cap Y\neq\emptyset$ and $\aff(X)=\aff(Y)$.
For any $z^0\in\re^n$ the sequence $(z^k)_{k\in\na}$
converges to some $\bar z\in X\cap Y$; if \ref{H1prime}
and~\ref{H2prime} hold at $\bar z$, the convergence is
superlinear.
\end{corollary}

\begin{proof}
By Lemma~\ref{lem:equivariance}, the affine isometry
$\phi:L\to\re^d$ sends $(z^k)_{k\geq 1}$ to the \cCRM\
sequence for $(\phi(X),\phi(Y))$ in $\re^d$.
Condition~\ref{H1prime} gives
$\inte(\phi(X)\cap\phi(Y))\neq\emptyset$, i.e.,~\ref{H1}
for the image pair.
Condition~\ref{H2prime} gives that $\phi(\partial_L X)$
and $\phi(\partial_L Y)$ are $\mathcal{C}^1$ manifolds of
dimension $d-1$ in $\re^d$, i.e.,~\ref{H2} for the image
pair.
Theorem~\ref{thm:orig} applies to the image sequence, which
converges superlinearly to $\phi(\bar z)$. Since $\phi^{-1}$
is also an isometry,
$\norm{z^k-\bar z}=\norm{\phi(z^k)-\phi(\bar z)}$ for all
$k$, so the ratios defining superlinear convergence are
identical for the two sequences and the rate transfers back
to $(z^k)_{k\in\na}$.
\end{proof}

\begin{theorem}[Q-quadratic convergence]\label{thm:main}
Let $X,Y\subset\re^n$ be closed convex with
$X\cap Y\neq\emptyset$ and $\aff(X)=\aff(Y)$.
For any $z^0\in\re^n$ the sequence $(z^k)_{k\in\na}$
converges to some $\bar z\in X\cap Y$; if \ref{H1prime}
and~\ref{H2pp} hold at $\bar z$, the convergence is
Q-quadratic, with
\begin{equation}\label{eq:quad_rate}
  \limsup_{k\to\infty}
  \frac{\dist(z^{k+1},X\cap Y)}{\dist(z^k,X\cap Y)^2}
  \leq\frac{2\max(\kappa_X,\kappa_Y)}{\omega},
\end{equation}
where $\kappa_X,\kappa_Y$ are the curvatures of
$\partial_L X$ and $\partial_L Y$ at $\bar z$ and $\omega$
is the regularity modulus of the pair at $\bar z$
in~\cref{eq:omega_def}.
\end{theorem}

\begin{proof}
By Lemma~\ref{lem:equivariance} the affine isometry
$\phi:L\to\re^d$ carries the iterates into $\re^d$, where
the image pair satisfies $\inte(\phi(X)\cap\phi(Y))\neq
\emptyset$ and has $\mathcal{C}^2$ relative boundaries of
dimension $d-1$. Distances, the error-bound modulus
$\omega$ of~\cref{eq:omega_def}, and the boundary curvatures
$\kappa_X,\kappa_Y$ of~\cref{eq:kappa_def} are all preserved
by $\phi$: distances because $\phi$ is an isometry, $\omega$
because both numerator and denominator in~\cref{eq:omega_def}
are distances, and the curvatures because $\kappa$ is an
intrinsic invariant of the hypersurface. It therefore
suffices to prove~\cref{eq:quad_rate} for the image pair in
$\re^d$ with the same constants, and we work there,
writing $X,Y$ for $\phi(X),\phi(Y)$ to lighten notation.

Since \ref{H2pp} implies \ref{H2prime},
Corollary~\ref{cor:super} gives superlinear convergence of
$(z^k)_{k\in\na}$ to $\bar z$, so
$\norm{T(z^k)-\bar z}=o(\dist(z^k,X\cap Y))$.
Fix $\varepsilon>0$; for all large $k$,
$\norm{T(z^k)-\bar z}\leq\varepsilon\,\dist(z^k,X\cap Y)$.
By~\cref{eq:fejer_bound} and Lemma~\ref{lem:background}(ii),
\[
  \norm{z_C^k-\bar z}\leq\norm{z^k-\bar z}
  \leq 2\,\dist(z^k,X\cap Y).
\]
The triangle inequality then gives
\[
  \norm{z_C^k-T(z^k)}
  \leq\norm{z_C^k-\bar z}+\norm{\bar z-T(z^k)}
  \leq(2+\varepsilon)\,\dist(z^k,X\cap Y).
\]
By Lemma~\ref{lem:background}(iii),
$T(z^k)=P_{H_X^k\cap H_Y^k}(z_C^k)$, where $H_X^k$ is
tangent to $\partial X$ at $P_X(z_C^k)$
by~\cite[Theorems~23.2 and~25.1]{Rockafellar:1970}.
For $k$ large enough $z_C^k\notin X$ (otherwise
$\dist(z^{k+1},X\cap Y)=0$ and the bound is trivial), so
$P_X(z_C^k)\in\partial X$; since $P_X(z_C^k)\to\bar z$ and
$\partial X$ is $\mathcal{C}^2$ near $\bar z$ by~\ref{H2pp},
$P_X(z_C^k)$ is a regular boundary point in the
neighborhood~$U$ of Lemma~\ref{lem:curvature} for all large
$k$.
Because $H_X^k$ supports $X$ at $P_X(z_C^k)$ with outward
normal $z_C^k-P_X(z_C^k)$, the point $P_X(z_C^k)$ is the
orthogonal projection of $z_C^k$ onto $H_X^k$; as
$T(z^k)\in H_X^k$, the Pythagorean identity
\[
  \norm{z_C^k-T(z^k)}^2
  =\norm{z_C^k-P_X(z_C^k)}^2+\norm{P_X(z_C^k)-T(z^k)}^2
\]
holds, whence
$\norm{T(z^k)-P_X(z_C^k)}\leq\norm{z_C^k-T(z^k)}$.
Fix $\eta>0$. Applying the uniform bound of
Lemma~\ref{lem:curvature} with $C=X$, the moving regular
point $p=P_X(z_C^k)$, and $w=T(z^k)\in H_X^k$ gives, for all
large $k$,
\[
  \dist(T(z^k),X)
  \leq\bigl(\tfrac{1}{2}\kappa_X+\eta\bigr)
       \norm{T(z^k)-P_X(z_C^k)}^2
  \leq\bigl(\tfrac{1}{2}\kappa_X+\eta\bigr)
       (2+\varepsilon)^2\dist(z^k,X\cap Y)^2,
\]
and symmetrically
$\dist(T(z^k),Y)\leq(\tfrac{1}{2}\kappa_Y+\eta)(2+\varepsilon)^2
\dist(z^k,X\cap Y)^2$.
Applying~\cref{eq:EB} gives
\[
  \dist(z^{k+1},X\cap Y)
  \leq\frac{(2+\varepsilon)^2
       \bigl(\tfrac{1}{2}\max(\kappa_X,\kappa_Y)+\eta\bigr)}
       {\omega}
       \,\dist(z^k,X\cap Y)^2.
\]
Since $\varepsilon>0$ and $\eta>0$ are
arbitrary,~\cref{eq:quad_rate} follows by taking $\limsup$.
\end{proof}

\begin{remark}\label{rem:rate}
The bound $2\max(\kappa_X,\kappa_Y)/\omega$
in~\cref{eq:quad_rate} is an upper estimate, obtained by
combining the factor $\tfrac12$ of
Lemma~\ref{lem:curvature} with the F\'{e}jer bound
$\norm{z^k-\bar z}\leq 2\,\dist(z^k,X\cap Y)$
in~\cref{eq:fejer_bound}. It is generally not tight: the
contraction $\norm{T(z^k)-P_X(z_C^k)}/\dist(z^k,X\cap Y)$ has
a limit strictly below the worst-case value $2+\varepsilon$
used in the proof (see Table~\ref{tab:numerics} and the
discussion in Example~\ref{ex:disks}), so we do not claim
the constant is sharp.
The F\'{e}jer factor of~\cref{eq:fejer_bound} can be dropped
whenever $\dist(z^k,X\cap Y)=\norm{z^k-\bar z}$ along the
trajectory, that is, whenever the nearest point of
$X\cap Y$ to each iterate is the limit $\bar z$ itself. This
holds when $\bar z$ is an isolated point of $X\cap Y$, and
also when $\bar z$ is a corner of $X\cap Y$ approached from
outside, as in Example~\ref{ex:disks}; in either case the
sharper bound $\tfrac12\max(\kappa_X,\kappa_Y)/\omega$
applies. An isolated intersection forces $d=2$ by a
dimension count, two relative boundaries of dimension
$d-1$ meeting transversally in $L$ intersect in a set of
dimension $d-2$, a point only when $d=2$, whereas for the
higher-dimensional examples of Section~\ref{sec:applications}
($d>2$), $X\cap Y$ has positive dimension and the general
bound $2\max(\kappa_X,\kappa_Y)/\omega$ applies.
The constant depends only on the local geometry at $\bar z$,
not on the ambient dimension or on how $L$ sits in $\re^n$.
For $d=n$, Theorem~\ref{thm:main} sharpens the superlinear
rate of~\cite[Theorem~5.3]{Behling:2024} to Q-quadratic
whenever the boundaries are $\mathcal{C}^2$.
As in Newton's method, the rate is governed by second-order
geometry, here boundary curvature $\kappa$, relative to a
first-order regularity constant $\omega$, despite \cCRM\
computing only projections and circumcenters.
\end{remark}

\section{Examples and applications}\label{sec:applications}

We illustrate the hypotheses and rates of
Corollary~\ref{cor:super} and Theorem~\ref{thm:main} with
concrete examples.
The first two are geometric; the remaining ones show that
the setting $\aff(X)=\aff(Y)\subsetneq\re^n$, where
Theorem~\ref{thm:orig} is silent, arises naturally in
structured feasibility problems.

\begin{example}[Overlapping discs in a plane]%
\label{ex:disks}
Let $L=\{(x,y,0)\in\re^3\}$ and set
\[
  X=\{(x,y,0)\in L:x^2+y^2\leq 4\},\quad
  Y=\bigl\{(x,y,0)\in L:(x-\sqrt{15})^2+y^2\leq 4\bigr\}.
\]
Both are discs of radius~$2$ in $L$, with centers
$\sqrt{15}\approx 3.87$ apart; they overlap in a thin lens
with nonempty $L$-interior.
Since $L$ has measure zero in $\re^3$,
$\inte_{\re^3}(X\cap Y)=\emptyset$ and
Theorem~\ref{thm:orig} does not apply.
The two boundary circles meet at $(\sqrt{15}/2,\pm1/2,0)$.

Take $\bar z=(\sqrt{15}/2,1/2,0)$.
The point $(\sqrt{15}/2,0,0)$ satisfies
$(\sqrt{15}/2)^2=15/4<4$, so it belongs to
$\operatorname{ri}(X)\cap\operatorname{ri}(Y)$,
giving~\ref{H1prime}.
Both relative boundaries are $\mathcal{C}^\infty$ circles,
so~\ref{H2prime} and~\ref{H2pp} hold.
The isometry $\phi(x,y,0)=(x,y)$ identifies the \cCRM\
sequence in $\re^3$ with \cCRM\ for two overlapping discs
in $\re^2$ with nonempty interior.
Since each circle of radius~$2$ has constant curvature
$\kappa=\tfrac{1}{2}$, the Q-quadratic constant is the same
for every limit point in $X\cap Y$, a consequence of the
circles having constant curvature.

Table~\ref{tab:numerics} reports five \cCRM\ steps from
$z^0=(\sqrt{15}/2,4,\tfrac{1}{2})$, which lies outside $L$
and outside both discs.
The first step projects onto $L$; all subsequent iterates
remain there. The quadratic ratio
$\dist(z^{k+1},X\cap Y)/\dist(z^k,X\cap Y)^2$ stabilizes at
$0.5551$ from $k=3$ onward, confirming Q-quadratic
convergence. This observed constant sits well below the
upper bound of Remark~\ref{rem:rate}. The limit $\bar z$ is
a corner of the lens $X\cap Y$, the transversal crossing of
the two circles; the iterates approach it from outside the
lens, so along the trajectory the nearest point of
$X\cap Y$ is $\bar z$ itself and
$\dist(z^k,X\cap Y)=\norm{z^k-\bar z}$ (this is what the
table tabulates). The relevant comparison is therefore the
sharper bound $\tfrac12\kappa/\omega$ of
Remark~\ref{rem:rate}, in which the F\'{e}jer factor is
absent. With $\kappa=\tfrac12$ and the regularity modulus
of~\cref{eq:omega_def} computed directly as
$\omega\approx0.252$, this bound is
$\tfrac12\kappa/\omega\approx0.99$. The observed $0.5551$
realizes a little over half of it, because the proof bounds
the contraction
$\norm{T(z^k)-P_X(z_C^k)}/\dist(z^k,X\cap Y)$ by its
worst-case value, whereas its true limit here is $0.742$.
The gap is an artifact of that estimate, not of the
iteration.

\begin{table}[ht]
\caption{\cCRM\ iterates for Example~\ref{ex:disks},
  $z^0=(\sqrt{15}/2,\,4,\,\tfrac{1}{2})$,
  $\bar z=(\sqrt{15}/2,\,\tfrac{1}{2},\,0)$. The $k=0$ step
  projects $z^0$ onto $L$ and is pre-asymptotic, so its
  quadratic ratio does not yet reflect the squared-error
  scaling; the ratio settles from $k=3$ on.}
\label{tab:numerics}
\renewcommand{\arraystretch}{1.2}
\begin{tabular}{ccccc}
\toprule
$k$ &
$\norm{z^k-\bar z}$ &
$\norm{z^{k+1}-\bar z}$ &
$\dfrac{\norm{z^{k+1}-\bar z}}{\norm{z^k-\bar z}}$ &
$\dfrac{\norm{z^{k+1}-\bar z}}{\norm{z^k-\bar z}^2}$ \\[5pt]
\midrule
$0$ & $3.54\times10^{\phantom{-}0}$ & $9.24\times10^{-2}$
    & $2.61\times10^{-2}$  & $0.007$ \\
$1$ & $9.24\times10^{-2}$ & $3.70\times10^{-3}$
    & $4.00\times10^{-2}$  & $0.433$ \\
$2$ & $3.70\times10^{-3}$ & $7.51\times10^{-6}$
    & $2.03\times10^{-3}$  & $0.549$ \\
$3$ & $7.51\times10^{-6}$ & $3.13\times10^{-11}$
    & $4.17\times10^{-6}$  & $0.555$ \\
$4$ & $3.13\times10^{-11}$ & $5.45\times10^{-22}$
    & $1.74\times10^{-11}$ & $0.555$ \\
\bottomrule
\end{tabular}
\end{table}
\end{example}

\begin{example}[Overlapping ellipses in a plane]%
\label{ex:ellipses}
Let $L=\{(x,y,0)\in\re^3\}$ and set
\[
  X=\Bigl\{(x,y,0)\in L:\tfrac{x^2}{4}+y^2\leq 1\Bigr\},\quad
  Y=\Bigl\{(x,y,0)\in L:(x-1)^2+\tfrac{y^2}{4}\leq 1\Bigr\}.
\]
Here $X$ has semi-axes $2$ and $1$ (along $x$ and $y$)
centered at the origin, and $Y$ has semi-axes $1$ and $2$
centered at $(1,0,0)$.
The point $(1/2,0,0)$ satisfies $1/16<1$ and $1/4<1$, so
it belongs to $\operatorname{ri}(X)\cap\operatorname{ri}(Y)$,
giving~\ref{H1prime}.
Both relative boundaries are $\mathcal{C}^\infty$ ellipses,
so~\ref{H2prime} and~\ref{H2pp} hold, and
Theorem~\ref{thm:main} applies.

The curvature of $\partial_L X$ at $(2\cos t,\sin t,0)$ is
\[
  \kappa_X(t)=\frac{2}{(4\sin^2\!t+\cos^2\!t)^{3/2}},
\]
which equals $2$ at the major-axis tips $(t=0,\pi)$ and
$\tfrac{1}{4}$ at the minor-axis tips
$(t=\tfrac{\pi}{2},\tfrac{3\pi}{2})$, an eightfold
variation.
Consequently, the Q-quadratic constant
$\max(\kappa_X,\kappa_Y)/\omega$ depends on the limit
point: different starting points lead to different asymptotic
rates, in contrast to Example~\ref{ex:disks}, where constant
curvature gives a universal rate.
\end{example}

The following example extends Example~4.12
of~\cite{Arefidamghani:2021}, where $Y$ is the $x$-axis,
to the halfplane $Y=\{y\leq 0\}$; it also probes the
boundary case $\beta=0$ and the open case
$\aff(X)\neq\aff(Y)$.

\begin{example}[Epigraph feasibility]\label{ex:epi}
Let $\alpha>1$, $\beta\geq 0$,
$\phi(t)=\lvert t\rvert^\alpha-\beta$, and
\[
  X \coloneqq \{(x,y)\in\re^2 :
               y \geq \lvert x\rvert^\alpha-\beta\},
  \quad
  Y \coloneqq \{(x,y)\in\re^2 : y \leq 0\}.
\]

When $\beta>0$, the intersection
$X\cap Y=\{(x,y):\lvert x\rvert^\alpha-\beta\leq y\leq 0\}$
is a compact lens with nonempty interior in $\re^2$,
so~\ref{H1} and~\ref{H1prime} both hold.
The boundary $\partial X=\{y=\lvert x\rvert^\alpha-\beta\}$
has second derivative
$\alpha(\alpha-1)\lvert x\rvert^{\alpha-2}$, bounded at
$x=0$ if and only if $\alpha\geq 2$.
For $\alpha\geq 2$, condition~\ref{H2pp} holds globally and
Theorem~\ref{thm:main} gives Q-quadratic convergence of
\cCRM.
For $1<\alpha<2$, the boundary is $\mathcal{C}^1$ but not
$\mathcal{C}^2$ at $x=0$, so~\ref{H2prime} holds
but~\ref{H2pp} fails at the vertex, and
Corollary~\ref{cor:super} gives superlinear convergence.
In both cases $\kappa_Y=0$ since $\partial Y=\{y=0\}$ is
flat. Note that for $\beta>0$ the lens has nonempty
interior, so a Slater point exists; the finite-termination
variant of~\cite{Behling:2024b} would then solve the
feasibility problem in finitely many steps. Plain \cCRM,
which we analyze here, does not terminate finitely, and the
Q-quadratic rate describes its genuine asymptotic behavior.

For $\beta=0$ the curve
$\partial X=\{y=\lvert x\rvert^\alpha\}$ meets $\{y=0\}$
tangentially at the origin, $X\cap Y=\{(0,0)\}$, and
$\operatorname{ri}(X)\cap\operatorname{ri}(Y)=\emptyset$,
so~\ref{H1prime} fails.
By~\cite{Arefidamghani:2021}, MAP converges sublinearly and
CRM linearly with constant $1-1/\alpha$.
We show that \cCRM\ also converges linearly with the same
constant.

For $z^k=(x_k,0)$ with $x_k>0$, the projection
$P_X((x_k,0))=(u_k,u_k^\alpha)$ has $u_k^\alpha>0$,
so the centralized point $z_C^k=(a_k,h_k)$ has $h_k>0$
and $R_Y(z_C^k)=(a_k,-h_k)$.
These two points are symmetric about $\{y=0\}$, so the
circumcenter satisfies $z^{k+1}=(x_{k+1},0)$ and the
sequence stays on $\{y=0\}$ by induction.

Equidistance of the circumcenter $(x_{k+1},0)$ from $z_C^k$
and $R_X(z_C^k)=2(p_k,p_k^\alpha)-z_C^k$ gives
\begin{equation}\label{eq:circ_formula}
  (x_{k+1}-p_k)(p_k-a_k) = p_k^\alpha(p_k^\alpha-h_k),
\end{equation}
where $p_k$ is the $x$-coordinate of $P_X(z_C^k)$,
$a_k=(u_k+v_k)/2$, and $h_k=v_k^\alpha/2$, with $v_k$ the
$x$-coordinate of $P_X((u_k,0))$.
The MAP normal equation $x=u(1+\alpha u^{2(\alpha-1)})$,
obtained by setting the derivative of
$(t-x)^2+t^{2\alpha}$ to zero at $t=u$, inverts by iterated
substitution ($u=x-\alpha x^{2\alpha-1}+\cdots$) to give
\[
  u_k=x_k-\alpha x_k^{2\alpha-1}+O(x_k^{4\alpha-3}),\quad
  v_k=x_k-2\alpha x_k^{2\alpha-1}+O(x_k^{4\alpha-3}),
\]
hence
$a_k=x_k-\tfrac{3\alpha}{2}x_k^{2\alpha-1}
+O(x_k^{4\alpha-3})$ and
$h_k=\tfrac{1}{2}x_k^\alpha+O(x_k^{3\alpha-2})$.
The normal equation
$a_k-p_k=\alpha p_k^{\alpha-1}(p_k^\alpha-h_k)$ yields
$p_k=x_k-2\alpha x_k^{2\alpha-1}+O(x_k^{4\alpha-3})$, so
$p_k-a_k=-\tfrac{\alpha}{2}x_k^{2\alpha-1}
+O(x_k^{4\alpha-3})$ and
$p_k^\alpha-h_k=\tfrac{1}{2}x_k^\alpha
+O(x_k^{3\alpha-2})$.
Substituting into~\cref{eq:circ_formula}:
\[
  x_{k+1}=p_k+\frac{p_k^\alpha(p_k^\alpha-h_k)}{p_k-a_k}
  =x_k\!\left(1-\frac{1}{\alpha}\right)+O(x_k^{2\alpha-1}),
\]
so $\lvert x_{k+1}\rvert/\lvert x_k\rvert\to 1-1/\alpha$.
A direct numerical iteration confirms this rate: starting
from $x_0=0.3$, the ratios converge to $0.5$ for $\alpha=2$
and to $0.6\overline{6}$ for $\alpha=3$, matching
$1-1/\alpha$ to all computed digits.
Table~\ref{tab:epi} summarizes all convergence rates.

\begin{table}[ht]
\caption{Convergence rates for Example~\ref{ex:epi}.
  MAP and CRM rates from~\cite{Arefidamghani:2021};
  \cCRM\ rates established above.}
\label{tab:epi}
\renewcommand{\arraystretch}{1.3}
\begin{tabular}{ccccc}
\toprule
$\beta$ & $\alpha$ & MAP
  & CRM~\cite{Arefidamghani:2021} & \cCRM \\
\midrule
$\beta=0$ & $\alpha>1$
  & sublinear & linear ($1-1/\alpha$) & linear ($1-1/\alpha$) \\
$\beta>0$ & $1<\alpha<2$
  & linear    & superlinear & superlinear${}^\dagger$ \\
$\beta>0$ & $\alpha\geq 2$
  & linear    & superlinear
  & \textbf{Q-quadratic}${}^{\dagger\dagger}$ \\
\bottomrule
\end{tabular}
\par\smallskip
{\footnotesize
${}^\dagger$~For the halfplane pair $\{X,Y\}$,~\cite{Behling:2024}
already gives superlinear via~\ref{H1}; for the line pair
$\{X,Y'\}$ below, $\inte(X\cap Y')=\emptyset$ so~\ref{H1}
fails and Corollary~\ref{cor:super} via~\ref{H1prime} is
required.\\[2pt]
${}^{\dagger\dagger}$~Q-quadratic is new for both pairs;
\cite{Behling:2024} establishes only superlinear under the
$\mathcal{C}^1$ assumption.}
\end{table}

Replace $Y$ by $Y'\coloneqq\{y=0\}$.
Then $\aff(Y')=\{y=0\}\neq\re^2=\aff(X)$, so neither
Corollary~\ref{cor:super} nor Theorem~\ref{thm:main} applies
directly.
Nevertheless, for any $z^0\in Y'$ the \cCRM\ sequences for
$\{X,Y\}$ and $\{X,Y'\}$ are identical:
$P_X((x_k,0))=(u_k,u_k^\alpha)$ has $u_k^\alpha>0$, so
$P_Y(P_X((x_k,0)))=P_{Y'}(P_X((x_k,0)))=(u_k,0)$; the
same holds for the second MAP step; the centralized point
$z_C^k=(a_k,h_k)$ has $h_k>0$, giving
$R_Y(z_C^k)=R_{Y'}(z_C^k)=(a_k,-h_k)$; and the circumcenter
lies on $\{y=0\}$ as shown above.
All rates in Table~\ref{tab:epi} therefore hold for
$\{X,Y'\}$, establishing superlinear and Q-quadratic
convergence of \cCRM\ even when $\aff(X)\neq\aff(Y)$.
\end{example}

The remaining examples show that
$\aff(X)=\aff(Y)\subsetneq\re^n$ is ubiquitous in
structured feasibility; in each case Theorem~\ref{thm:orig}
is inapplicable because equality constraints force $X\cap Y$
into a proper affine subspace.

\begin{example}[Equality-constrained smooth feasibility]%
\label{ex:eq_feas}
Given $A\in\re^{m\times N}$ of rank $m$, set $L=\{Az=b\}$,
$X=\{z\in L:g(z)\leq 0\}$, and $Y=\{z\in L:h(z)\leq 0\}$
for smooth convex $g,h:\re^N\to\re$.
Suppose a Slater point exists: some $z_0\in L$ with
$g(z_0)<0$ and $h(z_0)<0$. It lies in
$\operatorname{ri}(X)\cap\operatorname{ri}(Y)$, which
is~\ref{H1prime}, and makes $X$ and $Y$ full-dimensional in
$L$, so $\aff(X)=\aff(Y)=L$. For any $m\geq 1$ the
intersection has empty interior in the ambient space,
$\inte_{\re^N}(X\cap Y)=\emptyset$, so Theorem~\ref{thm:orig}
does not apply.
Condition~\ref{H2pp} holds when $g$ and $h$ are
$\mathcal{C}^2$ with nonzero gradient on the respective
boundaries within $L$; Theorem~\ref{thm:main} then gives
Q-quadratic convergence, compared with at most linear
convergence of MAP~\cite{Bauschke:1993}.
A canonical instance is the equality-constrained ellipsoidal
feasibility problem $Az=b$, $\|B_iz-c_i\|\leq r_i$,
$i=1,2$, which arises in the initialization phase of
interior-point methods for quadratically constrained
programs.
\end{example}

\begin{example}[SOCP feasibility with equality constraints]%
\label{ex:socp}
For the system $Az=b$ and $Cz+d\in\mathcal{K}$,
$\mathcal{K}=\{(t,u):\|u\|\leq t\}$, set $L=\{Az=b\}$ and
split the cone constraint into two convex sets within $L$.
The boundary of $\mathcal{K}$ is $\mathcal{C}^\infty$ away
from the apex, so~\ref{H2pp} holds at any limit point not
at the apex, and Theorem~\ref{thm:main} gives Q-quadratic
convergence.
\end{example}

\begin{example}[Semidefinite feasibility]\label{ex:sdp}
Consider
\begin{equation}\label{eq:sdp_feas}
  \text{find}\;\Sigma\in\mathbb{S}^n:\quad
  \mathcal{A}(\Sigma)=b,\quad
  \Sigma\succeq 0,\quad
  \norm{\Sigma-\hat\Sigma}_F\leq r,
\end{equation}
where $\mathcal{A}:\mathbb{S}^n\to\re^m$ is linear.
Set $L=\{\mathcal{A}(\Sigma)=b\}$,
$X=\mathbb{S}^n_+\cap L$, and
$Y=\{\|\Sigma-\hat\Sigma\|_F\leq r\}\cap L$.
Existence of $\Sigma_0\in L$ with $\Sigma_0\succ 0$ and
$\|\Sigma_0-\hat\Sigma\|_F<r$ gives~\ref{H1prime}, and the
same strictly feasible $\Sigma_0$ forces $X$ and $Y$ to be
full-dimensional in $L$, so $\aff(X)=\aff(Y)=L$. Since
$\inte_{\mathbb{S}^n}(X\cap Y)=\emptyset$,
Theorem~\ref{thm:orig} does not apply.
The boundary $\partial_L Y$ is $\mathcal{C}^\infty$ whenever
$\hat\Sigma\notin L$, which holds generically.
At a limit point $\bar\Sigma$ where
$\lambda_{\min}(\bar\Sigma)=0$ is a simple eigenvalue
(equivalently $\operatorname{rank}\bar\Sigma=n-1$, the
generic rank on $\partial\mathbb{S}^n_+$), the cone boundary
is locally a $\mathcal{C}^\infty$ hypersurface of
$\mathbb{S}^n$, with normal the spectral projector onto the
kernel of $\bar\Sigma$~\cite[Theorem~5]{Alizadeh:1997}. If
in addition this normal is not orthogonal to $V_L$, a
transversality condition that holds for all $L$ outside a
measure-zero set, then $\partial\mathbb{S}^n_+$ meets $L$
transversally, $\partial_L X$ is a $\mathcal{C}^\infty$
hypersurface of $L$ at $\bar\Sigma$, and~\ref{H2pp} holds;
Theorem~\ref{thm:main} then gives Q-quadratic convergence.
Problem~\cref{eq:sdp_feas} arises as the feasibility
subproblem in SDP solvers~\cite{Alizadeh:1997}, where a
strictly feasible starting point is required before the
main algorithm begins.
\end{example}

\begin{example}[Fixed-trace matrix feasibility]%
\label{ex:matrix}
In $\mathbb{S}^n$, consider
\[
  \text{find}\;\Sigma\in\mathbb{S}^n:\quad
  \operatorname{tr}(\Sigma)=1,\quad
  \lambda_{\max}(\Sigma)\leq a,\quad
  \norm{\Sigma-\hat\Sigma}_F\leq r.
\]
With $L=\{\operatorname{tr}(\Sigma)=1\}$,
$X=\{\lambda_{\max}\leq a\}\cap L$, and
$Y=\{\|\cdot-\hat\Sigma\|_F\leq r\}\cap L$, the trace
constraint confines both sets to
$L\subsetneq\mathbb{S}^n$, and
$\inte_{\mathbb{S}^n}(X\cap Y)=\emptyset$.
Assume a strictly feasible point: some $\Sigma_0\in L$ with
$\lambda_{\max}(\Sigma_0)<a$ and
$\|\Sigma_0-\hat\Sigma\|_F<r$. This gives~\ref{H1prime} and
makes $X,Y$ full-dimensional in $L$, so $\aff(X)=\aff(Y)=L$.
At a limit point $\bar\Sigma$ where the active constraint
$\lambda_{\max}(\bar\Sigma)=a$ is attained at a simple
eigenvalue, the spectral boundary
$\{\lambda_{\max}=a\}\cap L$ is $\mathcal{C}^2$ near
$\bar\Sigma$~\cite{Lewis:1996} (simplicity is generic), and
the Frobenius-ball boundary is $\mathcal{C}^\infty$; thus
\ref{H2pp} holds and Theorem~\ref{thm:main} gives
Q-quadratic convergence.
This problem arises in robust covariance estimation and
portfolio optimization~\cite{Ledoit:2004}, where the trace
constraint encodes a budget normalization.
\end{example}

\section{Numerical experiments}\label{sec:numerics}

We close with two experiments that put the rate of
Theorem~\ref{thm:main} to the test on the structured problems
of Section~\ref{sec:applications}, and that compare \cCRM\
against the method of alternating projections (MAP) and
Douglas--Rachford (DR) on the same instances. Both problems
live in a proper affine subspace, so the interior hypothesis
of~\cite{Behling:2024} fails and only the relative-interior
analysis of this paper applies. Throughout we measure
$\dist(z^k,X\cap Y)$ by projecting each iterate onto the
intersection with Dykstra's algorithm, and we report the
quadratic ratio
$\dist(z^{k+1},X\cap Y)/\dist(z^k,X\cap Y)^2$. The
experiments were implemented in
Julia~\cite{Bezanson:2017}; a companion repository provides
the Julia code port that reproduces
every number and figure below.\footnote{Available at
\url{https://github.com/yunierbello/ccrm-quadratic}.}

The first instance is the equality-constrained ellipsoidal
problem of Example~\ref{ex:eq_feas}, reduced to its
two-dimensional affine hull by the isometry of
Lemma~\ref{lem:equivariance}. We take two ellipses,
$X=\{(x/2.2)^2+y^2\leq 1\}$ and
$Y=\{(x-3)^2+(y/2.2)^2\leq 1\}$, whose interiors overlap in
a thin lens; from $z^0=(1.5,3.0)$ the iterates ride down to
the upper boundary crossing $\bar z\approx(2.017,0.400)$,
where the curvatures are $\kappa_X\approx 1.074$ and
$\kappa_Y\approx 0.215$. Table~\ref{tab:numerics2} shows the
quadratic ratio settling near $0.094$, comfortably below the
upper bound $\tfrac12\max(\kappa_X,\kappa_Y)/\omega\geq
0.537$ of Theorem~\ref{thm:main} (the inequality uses
$\omega\leq 1$); the gap is the looseness of the F\'{e}jer
factor discussed in Remark~\ref{rem:rate}. \cCRM\ reaches
$\dist<10^{-10}$ in three steps, whereas MAP has not done so
after twenty-five.

\begin{table}[ht]
\caption{\cCRM\ on the equality-constrained ellipsoidal
  problem of Section~\ref{sec:numerics}. The quadratic ratio
  stabilizes near $0.094$; the row $k=3$ is at machine
  precision, where the squared denominator underflows.}
\label{tab:numerics2}
\renewcommand{\arraystretch}{1.2}
\centering
\begin{tabular}{cccc}
\toprule
$k$ & $\dist(z^k,X\cap Y)$ & linear ratio & quadratic ratio \\
\midrule
$0$ & $2.65\times10^{\phantom{-}0}$  & $6.55\times10^{-3}$ & $0.0025$ \\
$1$ & $1.74\times10^{-2}$ & $1.74\times10^{-3}$ & $0.1000$ \\
$2$ & $3.01\times10^{-5}$ & $2.85\times10^{-6}$ & $0.0945$ \\
$3$ & $8.59\times10^{-11}$ & --- & --- \\
\bottomrule
\end{tabular}
\end{table}

The second instance is the fixed-trace spectral feasibility
problem of Example~\ref{ex:matrix} in $\mathbb{S}^8$, with
$L=\{\operatorname{tr}\Sigma=1\}$,
$X=\{\lambda_{\max}\leq 0.35\}\cap L$, and a Frobenius ball
$Y$ of radius $0.4$ about a fixed center on $L$. Projections
onto $X$ and $Y$ are computed by Dykstra over the spectral
(or ball) constraint and the trace hyperplane. From a
starting matrix on $L$ at distance $\approx 11.9$ from the
solution, \cCRM\ again converges Q-quadratically, with the
quadratic ratio stabilizing near $0.17$ and
$\dist<10^{-13}$ after three steps; MAP needs nine steps to
reach $10^{-10}$ and DR is far slower still.
Figure~\ref{fig:convergence} plots all three methods on both
problems and makes the separation plain: the \cCRM\ curve
bends downward, the signature of a quadratic rate, while
MAP is straight (linear) and DR is nearly flat.

\begin{figure}[ht]
\centering
\begin{subfigure}[t]{0.49\textwidth}
  \centering
  \includegraphics[width=\textwidth]{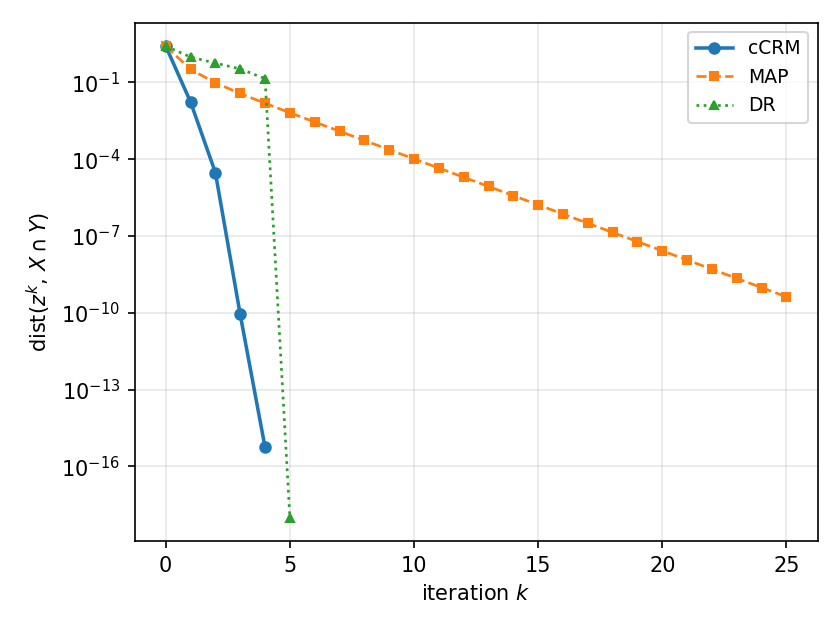}
  \caption{Equality-constrained ellipsoids in $\re^2$.}
  \label{fig:conv_ellipsoid}
\end{subfigure}
\hfill
\begin{subfigure}[t]{0.49\textwidth}
  \centering
  \includegraphics[width=\textwidth]{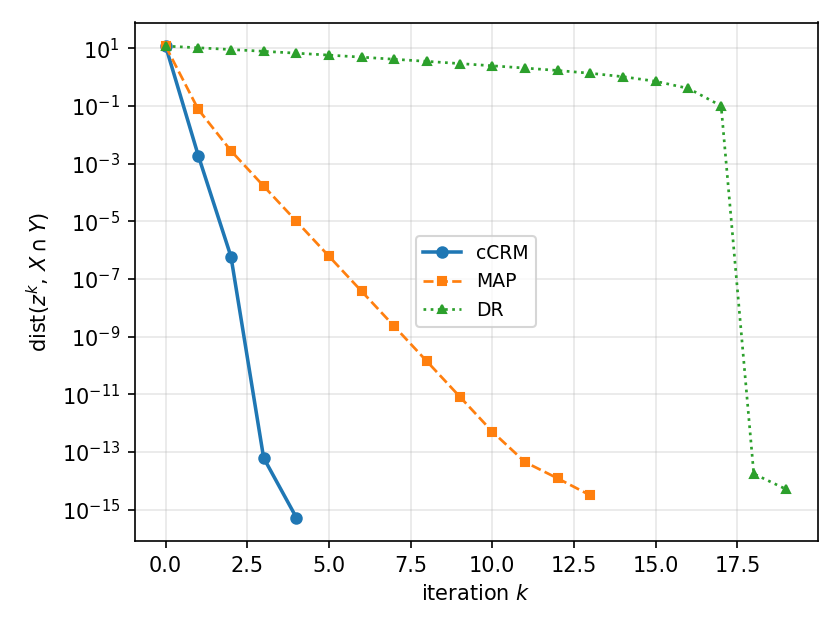}
  \caption{Fixed-trace spectral feasibility in $\mathbb{S}^8$.}
  \label{fig:conv_spectral}
\end{subfigure}
\caption{Distance to $X\cap Y$ versus iteration for \cCRM,
  MAP, and DR. On both problems \cCRM\ is Q-quadratic (the
  curve bends downward), MAP is linear (straight), and DR is
  slower. The vertical axis is logarithmic.}
\label{fig:convergence}
\end{figure}

These experiments confirm the analysis on two counts. The
quadratic ratios stabilize at finite values below the
predicted upper bound, as Theorem~\ref{thm:main} requires,
and the three-step descent to machine precision is exactly
the behavior a Q-quadratic rate produces. They also show that
the rate is not an asymptotic curiosity: on problems where
the classical interior hypothesis is unavailable, \cCRM\
solves the feasibility problem in a handful of steps where
the standard projection methods take many.

\section{Discussion}\label{sec:disc}

The proofs of Corollary~\ref{cor:super} and
Theorem~\ref{thm:main} rest on two ingredients.
First, when $\aff(X)=\aff(Y)=L$, the \cCRM\ iterates enter
$L$ after one step and remain there, so an isometry
$\phi:L\to\re^d$ reduces the problem to one with
$\inte(X\cap Y)\neq\emptyset$, where Theorem~\ref{thm:orig}
applies.
Second, Lemma~\ref{lem:curvature} gives
$\dist(T(z^k),X)=O(\dist(z^k,X\cap Y)^2)$, because
$T(z^k)$ lies on the tangent hyperplane to $\partial X$ at
$P_X(z_C^k)$ and the boundary has $\mathcal{C}^2$ contact
with that hyperplane; the error bound~\cref{eq:EB} then
converts this into~\cref{eq:quad_rate}.

The asymptotic constant in~\cref{eq:quad_rate} is local:
it depends only on $\kappa_X$, $\kappa_Y$, and $\omega$ at
$\bar z$, not on the global geometry of $X$ and $Y$ or on
how $L$ sits in $\re^n$.
Table~\ref{tab:numerics} gives an observed asymptotic
quadratic ratio $\approx 0.555$ for the constant-curvature
case $\kappa=1/2$, comfortably below the upper bound
$\tfrac12\kappa/\omega$ of Theorem~\ref{thm:main}; the gap
quantifies the looseness of the F\'{e}jer estimate used in
the proof.

The case $\aff(X)\neq\aff(Y)$ is open.
When the two affine hulls differ, the isometry reduction is
unavailable: $P_X(z^k)$ appears to acquire a component of
order $\norm{z^k-\bar z}$ in the direction transverse to
$\aff(Y)$, which the curvature argument cannot absorb,
because the second-order contact estimate of
Lemma~\ref{lem:curvature} controls only the component within
the common affine hull.
Example~\ref{ex:epi} shows that this obstruction is not
always fatal: for the pair $\{X,Y'\}$ the sequences for
$\{X,Y\}$ and $\{X,Y'\}$ coincide by a symmetry argument,
so all rates transfer without the isometry reduction.
Whether such a reduction exists more broadly, or whether a
different approach can establish superlinear or Q-quadratic
convergence for $\aff(X)\neq\aff(Y)$ in general, remains an
open problem.

\noindent {\bf Acknowledgements: }
The author thanks Roger Behling,
Alfredo N.~Iusem, and Luiz-Rafael Santos for many fruitful
discussions on the circumcentered-reflection method and its
variants~\cite{Behling:2018b,Behling:2018a,Arefidamghani:2021,Behling:2024} that directly motivated this work.

\bibliographystyle{spmpsci}

\end{document}